\documentclass{article}
\usepackage{amsmath,amssymb,amscd}
\usepackage{amsthm}
\newtheorem{thm}{Theorem}[section]
\newtheorem{prop}[thm]{Proposition}
\newtheorem{cor}[thm]{Corollary}
\newtheorem{rem}[thm]{Remark}

\newcommand{\spec}{\operatorname{Spec}}
\newcommand{\et}{\mathrm{\acute{e}t}}

\newcommand{\Zar}{\mathrm{Zar}}

\newcommand{\HO}{\operatorname{H}}

\usepackage[all]{xy}

\begin{document}

\author{Makoto Sakagaito}
\title{
A note on Gersten's conjecture 
for \'{e}tale cohomology
over two-dimensional
henselian regular local rings
}
\date{}
\maketitle
\begin{center}
Indian Institute of Science Education and Research, Mohali\footnote{
\textit{Present affiliation}:
Indian Institute of Science Education and Research, Bhopal.
\\
\textit{E-mail address}: makoto@iiserb.ac.in, sakagaito43@gmail.com
}
\end{center}

\begin{abstract}
We prove Gersten's conjecture
for \'{e}tale cohomology
over two dimensional 
henselian regular local rings
without
assuming equi-characteristic.
As an application, we obtain
the local-global principle 
for Galois cohomology
over mixed characteristic
two-dimensional henselian local rings.
\end{abstract}




\section{Introduction}
Let $R$ be an equi-characteristic regular local ring,
$k(R)$ the field of fractions of $R$,
$l$ a positive integer which is invertible in $R$ 
and $\mu_{l}$ the \'{e}tale sheaf of $l$-th roots of unity.
Then the sequence of \'{e}tale cohomology groups
\begin{align}\label{KC-I}
0
\to
\HO^{n+1}_{\et}
\left(
R,
\mu_{l}^{\otimes n}
\right)  
\to &
\HO^{n+1}_{\et}
\left(
k(R),
\mu_{l}^{\otimes n}
\right)   
\\
\to &
\displaystyle
\bigoplus_{
\substack
{
\mathfrak{p}\in \spec R  \nonumber \\
\operatorname{ht}(\mathfrak{p})
=
1
}
}
\HO^{n}_{\et}
\left(
\kappa(\mathfrak{p}),
\mu_{l}^{\otimes (n-1)}
\right) \nonumber \\
\to &
\displaystyle
\bigoplus_{
\substack
{
\mathfrak{p}\in \spec R  \nonumber \\
\operatorname{ht}(\mathfrak{p})
=
2
}
}
\HO^{n-1}_{\et}
\left(
\kappa(\mathfrak{p}),
\mu_{l}^{\otimes (n-2)}
\right)
\to \cdots \nonumber
\end{align}
is exact by  
Bloch-Ogus (\cite{B-O}) and Panin (\cite{P}). 
Here $\kappa(\mathfrak{p})$ is the residue 
field of $\mathfrak{p}\in \spec R$.

By using the exactness of the complex (\ref{KC-I}) 
at the first two terms, 
Harbater-Hartmann-Krashen (\cite{H-H-K}) and 
Hu (\cite{H}) proved the local-global
principle as follows.

\vspace{2.0mm}

Let $K$ be a field of one of the following types:
\begin{itemize}
\item[(a)] (semi-global case) The function field of a connected regular
projective curve over the field of fractions of
a henselian excellent discrete valuation ring
$A$.
\item[(b)] (local case) The function field of a two-dimensional
henselian excellent normal local domain $A$.
\end{itemize}
Then the following question was raised
by Colliot-Th\'{e}l\`{e}ne (\cite{C}): 

\vspace{2.0mm}

Let $n\geq 1$ be an integer and 
$l$ a positive integer which is invertible in $R$.
Is the natural map
\begin{equation}\label{CC}
 \operatorname{H}^{n+1}_{\et} 
 \left(
 K, \mu_{l}^{\otimes n}
 \right)
 \to
 \prod
 _{v\in \Omega_{K}}
 \operatorname{H}^{n+1}_{\et} 
 \left(
K_{v}, \mu_{l}^{\otimes n}
\right)
\end{equation}
injective ?

Here $\Omega_{K}$ is the set of normalized
discrete valuations on $K$
and 
$K_{v}$ is the corresponding henselization of $K$ for each 
$v\in\Omega_{K}$.

\vspace{2.0mm}

Suppose that
$A$ is equi-characteristic.
Harbater-Hartmann-Krashen (\cite[Theorem 3.3.6]{H-H-K}) proved that
the local-global map (\ref{CC}) is injective
in the semi-global case. 
Later, Hu (\cite[Theorem 2.5]{H}) proved that 
the local-global map (\ref{CC}) is injective
in both the semi-global case and the local case 
by an alternative method.

If the sequence (\ref{KC-I}) is exact (at the first two terms)
in the case where $R$ is
a mixed characteristic two-dimensional excellent
henselian local ring, then
the local-global map (\ref{CC}) is injective 
even without assuming equi-characteristic
(cf. \cite[Remark 3.3.7]{H-H-K} and
\cite[Remark 2.6 (2)]{H}).

In the case where $R$ is a local ring of 
a smooth algebra 
over a
(mixed characteristic) discrete valuation ring,
the sequence (\ref{KC-I}) is exact
(cf.\cite[Theorem 1.2 and Theorem 3.2 b)]{Ge}).

\vspace{2.0mm}

In this paper, we show the
following result:
\begin{thm}
(Theorem \ref{HenGerl})\label{MT}

Let $R$ be a mixed characteristic
two-dimensional excellent henselian local ring
and $l$ a positive integer which is invertible in $R$.
Then Gersten's conjecture 
for 
\'{e}tale cohomology
with
$\mu_{l}^{\otimes n}$ coefficients
holds over $\spec R$. 
That is, the sequence (\ref{KC-I}) is exact.
\end{thm}

See Remark \ref{Rem1} (iii) for the reason why we assume 
$\operatorname{dim}(R)=2$
in Theorem \ref{MT}.
We obtain the following result
as an application of Theorem \ref{MT}
:

\begin{thm}
\label{LG1}

With notations as above,
assume that $A$ is 
mixed characteristic and
$l$ is a positive integer 
which is invertible in $A$.

In both the semi-global case 
and the local case,
the local-global principle
for the Galois cohomology group
$\operatorname{H}^{n+1}(K, \mu_{l}^{\otimes n})$
holds for $n\geq 1$. 
That is,
the local-global map (\ref{CC})
is injective
for $n\geq 1$.
\end{thm}

V.Suresh also proved Theorem \ref{LG1} 
by an alternative method
(cf.\cite[Remark in Theorem 1.2]{H}).
\subsection{Notations}
For a scheme $X$,
$X^{(i)}$ is the set of points of
codimension $i$, 
$k(X)$ 
is 
the ring of rational functions
on $X$
and 
$\kappa(\mathfrak{p})$
is the residue field of 
$\mathfrak{p}\in X$.
If $X=\spec R$,
$k(\spec R)$ is abbreviated as $k(R)$. 
The symbol $\mu_{l}$ denotes 
the \'{e}tale sheaf of $l$-th roots of unity.
\section{Proof of the main result (Theorem \ref{MT})}
%
In this section, we use  the following results (Theorem \ref{DVR} 
and Theorem \ref{Pur}) repeatedly:
\begin{thm}\label{DVR}(cf.\cite[Theorem B.2.1 and Examples B.1.1.(2)]{CHK})
Let $A$ be a discrete valuation ring,
$K$ the function field of $A$ and
$l$ a positive integer which is invertible in $A$. Then the homomorphism
\begin{equation*}
\HO^{i}_{\et}(A, \mu_{l}^{\otimes n})
\to    
\HO^{i}_{\et}(K, \mu_{l}^{\otimes n})
\end{equation*}
is injective for any $i\geq 0$.
\end{thm}
\begin{thm}\label{Pur}(The absolute purity theorem \cite[p.159, Theorem 2.1.1]{G})
Let $Y\overset{i}{\hookrightarrow}X$ be a closed immersion of noetherian
regular schemes of pure codimension $c$. Let $n$ be an integer which is
invertible on $X$, and let $\Lambda=\mathbf{Z}/n$. Then the cycle class (cf.\cite[1.1]{G})
give an isomorphism
\begin{equation*}
\Lambda_{Y}
\xrightarrow{\sim} 
Ri^{!}\Lambda(c)[2c]
\end{equation*}
in $D^{+}(Y_{\et}, \Lambda)$. 
Here $D^{+}(Y_{\et}, \Lambda)$ is the derived category of complexes bounded below
of \'{e}tale sheaves of $\Lambda$-modules on $Y$.
\end{thm}
In this section, we use Theorem \ref{Pur} in the case where
$\operatorname{dim}X\leq 2$. In this case, Theorem \ref{Pur} was proved much earlier
by Gabber in 1976. See also \cite[\S 5, Remark 5.6]{SS} for a published proof.
\begin{prop}\label{HenInjl}
Let $R$ be a henselian regular local ring, 
$\mathfrak{m}$ the maximal ideal of $R$
and $K$ the function field of $R$.
Let $l$ be a positive integer such that 
$l\notin \mathfrak{m}$.  
Then
the homomorphism
\begin{equation}\label{inj}
\operatorname{H}^{i}_{\acute et}
\left(
\spec R,
\mu^{\otimes n}_{l}
\right)
\to
\operatorname{H}^{i}_{\acute et}
\left(
\spec K,
\mu^{\otimes n}_{l}
\right)
\end{equation}
is injective for any $i\geq 0$.
\end{prop}
\begin{proof}
We prove the statement by induction 
on $\operatorname{dim}(R)$.
Let $R$ be a discrete valuation ring 
(which does not need to be henselian).
Then the homomorphism (\ref{inj}) is injective by 
Theorem \ref{DVR}.

Assume that the statement is true for a henselian
regular
local ring of dimension $d$.

Let $R$ be a henselian regular local ring of dimension $d+1$,
$a\in \mathfrak{m}\setminus\mathfrak{m}^{2}$ and
$\mathfrak{p}=(a)$.
Then $R/\mathfrak{p}$ is a henselian regular local
ring of dimension $d$
and 
\begin{equation*}
k(R/\mathfrak{p})
=
R_{\mathfrak{p}}/\mathfrak{p}R_{\mathfrak{p}}
\end{equation*}
where $k(R/\mathfrak{p})$ is the function field of $R/\mathfrak{p}$.

Therefore the diagram
\begin{equation}\label{rid}
\begin{CD}
\operatorname{H}^{i}_{\acute et}
\left(
\spec R,
\mu^{\otimes n}_{l}
\right)
@>>>
\operatorname{H}^{i}_{\acute et}
\left(
\spec R_{\mathfrak{p}},
\mu^{\otimes n}_{l}
\right)\\
@VVV @VVV\\
\operatorname{H}^{i}_{\acute et}
\left(
\spec R/\mathfrak{p},
\mu^{\otimes n}_{l}
\right)
@>>>
\operatorname{H}^{i}_{\acute et}
\left(
\spec k(R/\mathfrak{p}),
\mu^{\otimes n}_{l}
\right)
\end{CD}
\end{equation}
is commutative. Then the left vertical map in the diagram 
(\ref{rid})
is an isomorphism by \cite[p.93, Theorem (4.9)]{A}
and the bottom horizontal map
in the diagram (\ref{rid})
is injective
by the induction hypothesis. 
Hence the homomorphism
\begin{equation*}
\operatorname{H}^{i}_{\acute et}
\left(
\spec R,
\mu^{\otimes n}_{l}
\right)
\to
\operatorname{H}^{i}_{\acute et}
\left(
\spec R_{\mathfrak{p}},
\mu^{\otimes n}_{l}
\right)
\end{equation*}
is  injective. Moreover the homomorphism
\begin{equation*}
\operatorname{H}^{i}_{\acute et}
\left(
\spec R_{\mathfrak{p}},
\mu^{\otimes n}_{l}
\right)
\to
\operatorname{H}^{i}_{\acute et}
\left(
\spec K,
\mu^{\otimes n}_{l}
\right)
\end{equation*}
is  injective by Theorem \ref{DVR}.
Therefore the statement follows.
\end{proof}
\begin{prop}(cf.\cite[Proposition 4.7]{S})\label{ex2}
Let $R$ be a regular local ring
and $l$ a positive integer which is 
invertible in $R$.
Suppose that $\operatorname{dim}(R)=2$.
Then the sequence
\begin{equation*}
\HO^{i}_{\et}(R, \mu_{l}^{\otimes n})
\to
\HO^{i}_{\et}(k(R), \mu_{l}^{\otimes n})
\xrightarrow{(*)}
\displaystyle
\bigoplus
_{\substack{
\mathfrak{p}\in 
\left(
\spec R
\right)^{(1)}} 
}
\HO^{i-1}_{\et}(\kappa(\mathfrak{p}), \mu_{l}^{\otimes (n-1)})
\end{equation*}
is exact for any $i\geq 0$.
\end{prop}
\begin{proof}
Let $A$ be a Dedekind ring, $\mathfrak{q}$
a maximal ideal of $A$.
Then 
\begin{equation*}
\HO^{i+1}_{\mathfrak{q}}((\spec A)_{\et}, \mu_{l}^{\otimes n})
=\HO^{i-1}_{\et}(\kappa(\mathfrak{q}), \mu_{l}^{\otimes (n-1)})
\end{equation*}
by Theorem \ref{Pur}.
Hence the sequence
\begin{equation*}
\HO^{i}_{\et}(A, \mu_{l}^{\otimes n})
\to
\HO^{i}_{\et}(U, \mu_{l}^{\otimes n})
\to
\bigoplus_{\mathfrak{q}\in Z^{(1)}}
\HO^{i-1}_{\et}(\kappa(\mathfrak{q}), \mu_{l}^{\otimes (n-1)})
\end{equation*}
is exact
where $Z$ is a closed subscheme of $\spec A$
and $U=\spec R\setminus Z$.
Since 
\begin{equation*}
\displaystyle\lim_{\substack{\to\\ U}}
\HO^{i}_{\et}(U, \mu_{l}^{\otimes n})
=\HO^{i}_{\et}(k(A), \mu_{l}^{\otimes n})
\end{equation*}
by \cite[pp.88--89, III, Lemma 1.16]{M}, the sequence
\begin{equation}\label{D1}
\HO^{i}_{\et}(A, \mu_{l}^{\otimes n})
\to
\HO^{i}_{\et}(k(A), \mu_{l}^{\otimes n})
\to
\displaystyle\bigoplus_{
\substack{
\mathfrak{q}\in 
\left(
\spec A
\right)^{(1)}}
}
\HO^{i-1}_{\et}(\kappa(\mathfrak{q}), \mu_{l}^{\otimes (n-1)})
\end{equation}
is exact.

Let $\mathfrak{m}$ be the maximal ideal of $R$.
Let
$g\in \mathfrak{m}\setminus \mathfrak{m}^{2}$,
$\mathfrak{p}=(g)$ and
\begin{math}
Z=\spec R/\mathfrak{p}.
\end{math}
Then $R/\mathfrak{p}$ is a regular local ring and 
we have
\begin{equation*}
\HO^{i+1}_{Z}((\spec R)_{\et}, \mu_{l}^{\otimes n})
=
\HO^{i-1}_{\et}(R/\mathfrak{p}, \mu_{l}^{\otimes (n-1)})
\end{equation*}
by Theorem \ref{Pur}. 

We consider the
commutative diagram
\begin{equation}\label{DIA}
\xymatrix@=10pt{
&
\HO^{i}_{\et}(R, \mu_{l}^{\otimes n})
\ar[r]\ar[d] &
\HO^{i}_{\et}(R_{g}, \mu_{l}^{\otimes n})
\ar[r]\ar[d] &
\HO^{i-1}_{\et}(R/\mathfrak{p}, \mu_{l}^{\otimes (n-1)})^{\prime}
\ar[d]\ar[r] & 0 \\
0\ar[r] &
\operatorname{Ker}(*)
\ar[r] &
\HO^{i}_{\et}(R_{g}, \mu_{l}^{\otimes n})^{\prime}
\ar[r] &
\HO^{i-1}_{\et}(k(R/\mathfrak{p}), \mu_{l}^{\otimes (n-1)}) \\
}
\end{equation}
where
\begin{align*}
\HO^{i-1}_{\et}(R/\mathfrak{p}, \mu_{l}^{\otimes (n-1)})^{\prime} 
=
\operatorname{Im}
\left(
\HO^{i}_{\et}(R_{g}, \mu_{l}^{\otimes n})
\to
\HO^{i-1}_{\et}(R/\mathfrak{p}, \mu_{l}^{\otimes (n-1)})
\right)
\end{align*}
and
\begin{align*}
&\HO^{i}_{\et}(R_{g}, \mu_{l}^{\otimes n})^{\prime}\\
=&
\operatorname{Ker}
\left(
\HO^{i}_{\et}(k(R_{g}), \mu_{l}^{\otimes n})
\to
\bigoplus_{\substack{
\mathfrak{q}\in 
\left(
\spec R_{g}
\right)
^{(1)}
}
}
\HO^{i-1}_{\et}(\kappa(\mathfrak{q})), \mu_{l}^{\otimes (n-1)})
\right).
\end{align*}
Then
the rows in the diagram (\ref{DIA}) are exact
by Theorem \ref{Pur}.  
Since $R_{g}$ is a Dedekind domain,
the middle map in the diagram (\ref{DIA}) is surjective
by (\ref{D1}). Moreover, since 
\begin{equation*}
\HO^{i-1}_{\et}(R/\mathfrak{p}, \mu_{l}^{\otimes (n-1)})^{\prime}
\subset 
\HO^{i-1}_{\et}(R/\mathfrak{p}, \mu_{l}^{\otimes (n-1)})
\end{equation*}
and $R/\mathfrak{p}$ is a discrete
valuation ring,
the right map in the diagram (\ref{DIA}) is injective
by Theorem \ref{DVR}. 
Therefore 
the statement follows from the snake lemma.
\end{proof}
\begin{cor}
Let $R$ be
the henselization of
a regular local ring
which is 
essentially of finite type
over a mixed characteristic discrete valuation ring.
Suppose that $\operatorname{dim}(R)=2$.
Then
\begin{equation*}
\operatorname{H}^{n+1}_{\Zar}
\left(
R, \mathbb{Z}/l(n)
\right)
=
0
\end{equation*}
for a positive integer $l$ which is
invertible in $R$.
Here 
$\mathbb{Z}(n)$ is Bloch's cycle complex and
$\mathbb{Z}/l(n)=\mathbb{Z}(n)\otimes\mathbf{Z}/l$ (cf. \cite[p.779]{Ge}).
\end{cor}
\begin{proof}
Let 
$\mathfrak{m}$
be the maximal ideal of $R$.
Let $g\in \mathfrak{m}\setminus \mathfrak{m}^{2}$
and $\mathfrak{p}=(g)$. 
Then the homomorphism
\begin{equation*}
\operatorname{H}^{n+1}_{\et}(R, \mu_{l}^{\otimes n})
\to
\operatorname{H}^{n+1}_{\et}(R_{g}, \mu_{l}^{\otimes n})
\end{equation*}
is injective by Proposition \ref{HenInjl}. 
Hence the homomorphism 
\begin{equation*}
\operatorname{H}^{n}_{\et}(R_{g}, \mu_{l}^{\otimes n})
\to
\operatorname{H}^{n-1}_{\et}(R/\mathfrak{p}, \mu_{l}^{\otimes n-1})
\end{equation*}
is surjective by Theorem \ref{Pur}. Therefore
the homomorphism 
\begin{equation*}
\operatorname{H}^{n}_{\Zar}(R_{g}, \mathbb{Z}/l(n))
\to
\operatorname{H}^{n-1}_{\Zar}(R/\mathfrak{p}, 
\mathbb{Z}/l(n-1))
\end{equation*}
is surjective by \cite[p.774, Theorem 1.2]{Ge}
and \cite{V}.
Moreover the homomorphism 
\begin{equation*}
\operatorname{H}^{n+1}_{\Zar}(R, \mathbb{Z}/l(n))
\to
\operatorname{H}^{n+1}_{\Zar}(R_{g}, \mathbb{Z}/l(n))
\end{equation*}
is injective
by the localization theorem \cite[p.779, Theorem 3.2]{Ge}. 
We consider
the commutative diagram
\begin{equation}\label{Rg}
\begin{CD}
\operatorname{H}^{n+1}_{\Zar}(R_{g}, \mathbb{Z}/l(n))
@>>>
\operatorname{H}^{n+1}_{\et}(R_{g}, \mathbb{Z}/l(n)) \\
@VVV @VVV \\
\operatorname{H}^{n+1}_{\Zar}(k(R_{g}), \mathbb{Z}/l(n))
@>>>
\operatorname{H}^{n+1}_{\et}(k(R_{g}), \mathbb{Z}/l(n)).
\end{CD}
\end{equation}
Then the upper map   in the commutative
diagram (\ref{Rg}) is injective
by the Beilinson-Lichenbaum conjecture
(\cite[p.774, Theorem 1.2]{Ge}, \cite{V})
and the right map 
in the commutative
diagram (\ref{Rg})
is injective 
by the commutative diagram (\ref{DIA})
in the proof of Proposition \ref{ex2}.
Hence the homomorphism
\begin{equation*}
\operatorname{H}^{n+1}_{\Zar}(R_{g}, \mathbb{Z}/l(n))
\to
\operatorname{H}^{n+1}_{\Zar}(k(R_{g}), \mathbb{Z}/l(n))
\end{equation*}
is injective and
the homomorphism
\begin{equation*}
\operatorname{H}^{n+1}_{\Zar}(R, \mathbb{Z}/l(n))
\to
\operatorname{H}^{n+1}_{\Zar}(k(R_{g}), \mathbb{Z}/l(n))
\end{equation*}
is also injective. Since
\begin{equation*}
\operatorname{H}^{n+1}_{\Zar}(k(R_{g}), \mathbb{Z}/l(n))
=0,
\end{equation*}
we have
\begin{equation*}
\operatorname{H}^{n+1}_{\Zar}(R, \mathbb{Z}/l(n))
=0.
\end{equation*}
This completes the proof.
\end{proof}
\begin{rem}\label{Rem1}
\begin{itemize}
\item[(i)]
If $R$ is a local ring of 
a smooth algebra
over a discrete valuation ring, then
\begin{equation*}
\operatorname{H}_{\Zar}^{i}
(R, \mathbb{Z}/m(n))=0
\end{equation*}
for $i> n$ and any positive integer $m$
(cf.\cite[p.786, Corollary 4.4]{Ge}).
\item[(ii)]
If we have
\begin{equation*}
\HO^{n+1}_{\Zar}(R, \mathbb{Z}/l(n))=0
\end{equation*}
for any regular local ring $R$
which is finite type over a discrete valuation ring
and a positive integer $l$ which is invertible in $R$,
we can show that the homomorphism
\begin{equation*}
\HO_{\et}^{n+1}(R, \mu_{l}^{\otimes n}) 
\to
\HO_{\et}^{n+1}(k(R), \mu_{l}^{\otimes n})
\end{equation*}
is injective by a similar augument as in the proof
of \cite[Theorem 4.2]{S2}.
\item[(iii)]
The reason why we assume 
$\operatorname{dim}(R)=2$ in Propositin \ref{ex2} and Theorem \ref{HenGerl}
is that we have to show that the middle map in the diagram (\ref{DIA}), i.e.,
the homomorphism
%
\begin{equation*}
\HO^{n+1}_{\et}(R_{g}, \mu_{l}^{\otimes n})   
\to
\HO^{n+1}_{\et}(R_{g}, \mu_{l}^{\otimes n})^{\prime} 
\end{equation*}
%
is surjective for 
an element $g$ of $\mathfrak{m}\setminus \mathfrak{m}^{2}$.
Here $\mathfrak{m}$ is the maximal ideal of $R$ and
\footnotesize
\begin{equation*}
\HO^{n+1}_{\et}(R_{g}, \mu_{l}^{\otimes n})^{\prime} 
=
\operatorname{Ker}\left(
\HO^{n+1}_{\et}(k(R), \mu_{l}^{\otimes n})
\to
\bigoplus_{\mathfrak{q}\in (\spec R_{g})^{(1)}}
\HO^{n}_{\et}(\kappa(\mathfrak{q}), \mu_{l}^{\otimes (n-1)})
\right).
\end{equation*}
\normalsize
If we have
\begin{equation*}
\HO^{n+1}_{\Zar}(R, \mathbb{Z}/l(n))
=
\HO^{n+2}_{\Zar}(R, \mathbb{Z}/l(n))=0
\end{equation*}
for any regular local ring $R$ which is 
finite type over a discrete valuation ring and 
a positive integer $l$ which is invertible in $R$, then  
\begin{equation*}
\HO^{n+1}_{\Zar}(R_{g}, \mathbb{Z}/l(n))
=
\HO^{n+2}_{\Zar}(R_{g}, \mathbb{Z}/l(n))=0
\end{equation*}
by the localization theorem (\cite[p.779, Theorem 3.2]{Ge}) and we can show that
%
\begin{align*}
\HO^{n+1}_{\et}(R_{g}, \mu_{l}^{\otimes n}) 
=\Gamma(\spec R_{g}, R^{n+1}\epsilon_{*}(\mu_{l}^{\otimes n})) 
=\HO^{n+1}_{\et}(R_{g}, \mu_{l}^{\otimes n})^{\prime} 
\end{align*}
%
and Proposition \ref{ex2} holds. Here 
$\epsilon: (\spec R_{g})_{\et}\to (\spec R_{g})_{\Zar}$ 
is the change of site maps.
\end{itemize}
\end{rem}
\begin{thm}\label{HenGerl}
Let $R$ be
a henselian regular local ring
with $\operatorname{dim}(R)=2$
and $l$
a positive integer which is invertible in $R$.
Then
the sequence
\begin{align}\label{2Gers}
0
\to
\HO^{i}_{\et}
\left(
R,
\mu_{l}^{\otimes n}
\right)  
\to
\HO^{i}_{\et}
\left(
k(R),
\mu_{l}^{\otimes n}
\right)  
\to
\displaystyle
\bigoplus_{
\substack
{
\mathfrak{p}\in \left(\spec R\right)^{(1)}  
}
}
\HO^{i-1}_{\et}
\left(
\kappa(\mathfrak{p}),
\mu_{l}^{\otimes (n-1)}
\right) 
\nonumber
\\
\to
\displaystyle
\bigoplus_{
\substack
{
\mathfrak{p}\in 
\left(\spec R \right)
^{(2)} 
}
}
\HO^{i-2}_{\et}
\left(
\kappa(\mathfrak{p}),
\mu_{l}^{\otimes (n-2)}
\right) \to 0
\end{align}
is exact for any $i\geq 0$.
\end{thm}
\begin{proof}

The exactness 
of the complex (\ref{2Gers})
at the first two terms
follows from Proposition \ref{HenInjl} and
Proposition \ref{ex2}.

We consider the coniveau spectral sequence
\begin{equation*}
\operatorname{E}^{p, q}_{1}
=
\coprod_{x\in (\spec R)^{(p)}}
\HO^{p+q}_{x}
\left(
\spec R, \mu^{\otimes n}_{l}
\right)
\Rightarrow   
\HO^{p+q}_{\et}
\left(
R, \mu^{\otimes n}_{l}
\right)
=\operatorname{E}^{p+q}
\end{equation*}
(cf.\cite[\S 1]{CHK}). Then we have a filtration
\begin{align*}
0\subset 
\operatorname{F}^{p+q}_{p+q}
\subset
\cdots
\subset
\operatorname{F}^{p+q}_{1}
\subset
\operatorname{F}^{p+q}_{0}
=
\operatorname{E}^{p+q},
\end{align*}
such that
\begin{equation*}
\operatorname{F}^{p+q}_{p}/
\operatorname{F}^{p+q}_{p+1}
\simeq
\operatorname{E}^{p, q}_{\infty}.
\end{equation*}
By Theorem \ref{Pur}, 
it suffices to show that
\begin{equation*}
\operatorname{E}_{2}^{1, i-1}
=
\operatorname{E}_{2}^{2, i-2}
=
0.
\end{equation*}
By Proposition \ref{HenInjl},
the morphism 
\begin{equation*}
\operatorname{E}^{i}
\to
\operatorname{E}^{0, i}_{\infty}
\end{equation*}
is injective and 
\begin{equation*}
\operatorname{F}^{i}_{1}
=
\operatorname{F}^{i}_{2}
=
0.
\end{equation*}
Hence we have
\begin{equation*}
\operatorname{E}_{\infty}^{1, i-1}
=
\operatorname{E}_{\infty}^{2, i-2}
=
0.
\end{equation*}
Since
\begin{equation*}
\operatorname{E}^{p, i-p+1}_{r}=0
\end{equation*}
for $p\geq 3$ and
\begin{equation*}
\operatorname{E}^{1-r, i+r-2}_{r}=0    
\end{equation*}
for $r\geq 2$, we have
\begin{equation*}
\operatorname{E}_{2}^{1, i-1}
=
\operatorname{E}_{\infty}^{1, i-1}
=0.
\end{equation*}
By the exactness of the complex (\ref{2Gers}) at the second term, we have
\begin{equation*}
\operatorname{E}^{0, i-1}_{2}
=
\operatorname{E}^{0, i-1}_{\infty}
=
\operatorname{E}^{i-1}
\end{equation*}
and
\begin{equation*}
\operatorname{Im}
\left(
\operatorname{E}^{0, i-1}_{2}
\xrightarrow{\operatorname{d}_{2}^{0, i-1}}
\operatorname{E}^{2, i-2}_{2}
\right)
=0.   
\end{equation*}
Hence we have
\begin{equation*}
\operatorname{E}^{2, i-2}_{2}
=
\operatorname{E}^{2, i-2}_{3}. 
\end{equation*}
Moreover, since
\begin{equation*}
\operatorname{E}^{2-r, i+r-3}_{r}=0    
\end{equation*}
for $r\geq 3$,
we have
\begin{equation*}
\operatorname{E}^{2, i-2}_{r+1} 
=
\frac{\operatorname{Ker}(\operatorname{d}_{r}^{2, i-2})}
{\operatorname{Im}(\operatorname{d}_{r}^{2-r, i+r-3})}
=
\operatorname{E}^{2, i-2}_{r}
\end{equation*}
for $r\geq 3$.
Therefore
\begin{equation*}
\operatorname{E}^{2, i-2}_{2}
=
\operatorname{E}^{2, i-2}_{3}
=
\operatorname{E}^{2, i-2}_{\infty}
=0.
\end{equation*}
This completes the proof.
\end{proof}
%


\end{document}